\documentclass[a4paper,12pt]{article}
\usepackage{float,amsthm,amsmath,amssymb}
\usepackage{url,color}
\theoremstyle{plain}
\newtheorem{thm}{Theorem}
\newtheorem{prop}[thm]{Proposition}

\newtheorem{clm}[thm]{Claim}
\theoremstyle{definition}
\newtheorem{defn}[thm]{Definition}

\newtheorem{rem}{Remark}

\title{Characterization and chromatic number of triangle-free graphs with diameter 2}
\author{
Akihiro Higashitani\thanks{Osaka University, Osaka, Japan,
E-mail: {\tt higashitani@ist.osaka-u.ac.jp}},
Diogo Kendy Matsumoto\thanks{
Teikyo University of Science, Yamanashi, Japan,
E-mail: {\tt diogo-swm@ntu.ac.jp}
} 
and
Naoki Matsumoto\thanks{University of the Ryukyus, Okinawa, Japan, 
E-mail: {\tt naoki.matsumo10@gmail.com}}}
\date{} 

\begin{document}

\maketitle

\begin{abstract}
In this paper, we consider triangle-free graphs with diameter 2.
If a triangle-free graph $G$ with diameter 2 is not isomorphic to a star,
then the radius of $G$ is also 2,
where such a graph is also called a {\it $2$-self-centered graph}.
Shekarriz et al.~[A characterization for 2-self-centered graphs,
{\it Discuss. Math. Graph Theory} {\bf 38} (2018), 27--37.]
gave a characterization of 2-self-centered graphs.
However, there is a slight flaw in their characterization.
Thus, in this paper, we modify it and prove an accurate characterization of those graphs.
Furthermore,
by using our characterization,
we prove some results concerning the chromatic number of triangle-free graphs with diameter 2.
\end{abstract}



\section{Introduction}

In this paper, we consider only finite simple graphs.
For a graph $G$ and a vertex $v$ of $G$,
the eccentricity of $v$ in $G$ is
the maximum graph distance between $v$ and any other vertex of $G$.
(If $G$ is disconnected, all vertices are defined to have infinite eccentricity.)
The \emph{diameter} (resp., \emph{radius}) of $G$ is 
the maximum (resp., minimum) eccentricity in $G$.

There are many interesting and important triangle-free graphs with diameter 2.
{\it Moore graphs} are the most famous family of such graphs,
where a Moore graph is a regular graph whose girth is more than twice of diameter.
The smallest Moore graph is the 5-cycle, 
and the next one is the Petersen graph.
Hoffman and Singleton~\cite{HS1960} proved that 
every graph with diameter 2 and girth 5 
is a $k$-regular graph with $k^2+1$ vertices,
where $k \in \{2,3,7,57\}$.
It is the most famous open problem
whether there exists a 57-regular graph with 3250 vertices, diameter 2 and girth 5.
Many studies have been conducted on triangle-free graphs with diameter 2,
even in recent years.
For example,
extremal problem~\cite{DKMCRVWW2025},
domination-like parameter~\cite{AK2024,CDMSS2023} and a special vertex coloring~\cite{FGO2021}.

It is easy to see that
a triangle-free graph with diameter 2 and radius 1 is only a star.
Thus, 
every triangle-free graph $G$ with diameter 2 has radius 2 unless $G$ is a star,
where such a graph is also called a triangle-free {\it $2$-self-centered graph}.
(In general, 
a graph whose diameter and radius are both $k$ is called a $k$-self-centered graph.)
Similarly to Moore graphs,
there are a number of studies on self-centered graphs;
see~\cite{AAA1984, Buckley1989, HM2018, IM2022, Janakiraman1994, NX1986, Stanic2010}.
In particular,
Shekarriz et al.~\cite{SMM2018} gave a characterization of triangle-free 2-self-centered graphs
in terms of a pair of coverings by independent sets.
However, their characterization needs some modification;
an $X$-generalized complete bipartite graph is not well defined,
which will be described in detailed in the next section (see Remark~\ref{rem:02}).
Thus, we modify them and prove the following,
where the definition of notations is introduced in the next section.

\begin{thm}[Theorem 13 in \cite{SMM2018}]\label{thm:ch}
A graph $G$ is a triangle-free $2$-self-centered graph
if and only if 
there are non-negative integers $k, \ell, r, s$ and a graph $X$
which has an SBIC via $(\mathbb{A}_r, \mathbb{B}_s)$
such that $G = GCB_X(k,\ell,\mathbb{A}_r, \mathbb{B}_s)$.
\end{thm}

We also have the following result on $2$-self-centered graphs with girth 5,
i.e., Moore graphs.

\begin{prop}\label{prop:girth5}
Let $G$ be a $2$-self-centered graph of girth $5$.
Then there are non-negative integers $r, s$ and a graph $X$
which has an SBIC via $(\mathbb{A}_r, \mathbb{B}_s)$
such that $G = GCB_X(1,1,\mathbb{A}_r, \mathbb{B}_s)$.
\end{prop}

Moreover, 
by using our correct characterization,
we characterize triangle-free {\it planar} 2-self-centered graphs,
as follows.

\begin{thm}\label{thm:planar}
A triangle-free $2$-self-centered graph $G$ is planar 
if and only if 
$G$ is isomorphic to one of the following:
\begin{itemize}
\item[(1)] A complete bipartite graph $K_{2,d}$ for $d \ge 2$.
\item[(2)] A \emph{blow-up} of $C_5 = u_0u_1u_2u_3u_4$ at $\{u_1,u_3\}$, 
that is,
$u_1$ and $u_3$ are replaced with independent sets $V_1$ and $V_3$ with size at least one
so that each vertex in $V_i$ is adjacent to both $u_{i-1}$ and $u_{i+1}$
for each $i \in \{1,3\}$ (see Figure~$\ref{fig:blowup_c5}$).
\end{itemize}
\end{thm}

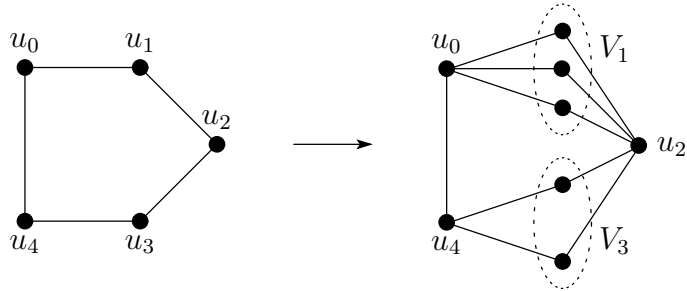
\begin{figure}[htb]
\centering
\unitlength 0.1in
\begin{picture}( 34.6000, 14.8000)( 13.7500,-21.5000)
%
\special{pn 8}%
\special{sh 1.000}%
\special{ar 1600 1000 40 40  0.0000000 6.2831853}%
%
\special{pn 8}%
\special{sh 1.000}%
\special{ar 1600 1800 40 40  0.0000000 6.2831853}%
%
\special{pn 8}%
\special{sh 1.000}%
\special{ar 2200 1000 40 40  0.0000000 6.2831853}%
%
\special{pn 8}%
\special{sh 1.000}%
\special{ar 2200 1800 40 40  0.0000000 6.2831853}%
%
\special{pn 8}%
\special{sh 1.000}%
\special{ar 2600 1400 40 40  0.0000000 6.2831853}%
%
\special{pn 8}%
\special{pa 1600 1800}%
\special{pa 1600 1000}%
\special{fp}%
%
\special{pn 8}%
\special{pa 1600 1000}%
\special{pa 2200 1000}%
\special{fp}%
%
\special{pn 8}%
\special{pa 2200 1000}%
\special{pa 2600 1400}%
\special{fp}%
%
\special{pn 8}%
\special{pa 2600 1400}%
\special{pa 2200 1800}%
\special{fp}%
%
\special{pn 8}%
\special{pa 2200 1800}%
\special{pa 1600 1800}%
\special{fp}%
%
\special{pn 8}%
\special{pa 3000 1400}%
\special{pa 3400 1400}%
\special{fp}%
\special{sh 1}%
\special{pa 3400 1400}%
\special{pa 3334 1380}%
\special{pa 3348 1400}%
\special{pa 3334 1420}%
\special{pa 3400 1400}%
\special{fp}%
\put(16.0000,-8.7000){\makebox(0,0){$u_0$}}%
\put(22.0000,-8.7000){\makebox(0,0){$u_1$}}%
\put(26.0000,-12.7000){\makebox(0,0){$u_2$}}%
\put(22.0000,-19.2000){\makebox(0,0){$u_3$}}%
\put(16.0000,-19.2000){\makebox(0,0){$u_4$}}%
%
\special{pn 8}%
\special{sh 1.000}%
\special{ar 3796 1006 40 40  0.0000000 6.2831853}%
%
\special{pn 8}%
\special{sh 1.000}%
\special{ar 3796 1806 40 40  0.0000000 6.2831853}%
%
\special{pn 8}%
\special{sh 1.000}%
\special{ar 4396 1006 40 40  0.0000000 6.2831853}%
%
\special{pn 8}%
\special{sh 1.000}%
\special{ar 4796 1406 40 40  0.0000000 6.2831853}%
%
\special{pn 8}%
\special{pa 3796 1806}%
\special{pa 3796 1006}%
\special{fp}%
%
\special{pn 8}%
\special{pa 3796 1006}%
\special{pa 4396 1006}%
\special{fp}%
%
\special{pn 8}%
\special{pa 4396 1006}%
\special{pa 4796 1406}%
\special{fp}%
\put(37.9500,-8.7500){\makebox(0,0){$u_0$}}%
\put(49.7000,-14.1000){\makebox(0,0){$u_2$}}%
\put(37.9500,-19.2500){\makebox(0,0){$u_4$}}%
%
\special{pn 8}%
\special{sh 1.000}%
\special{ar 4400 2010 40 40  0.0000000 6.2831853}%
%
\special{pn 8}%
\special{sh 1.000}%
\special{ar 4400 1610 40 40  0.0000000 6.2831853}%
%
\special{pn 8}%
\special{sh 1.000}%
\special{ar 4400 1210 40 40  0.0000000 6.2831853}%
%
\special{pn 8}%
\special{pa 4400 810}%
\special{pa 4800 1410}%
\special{fp}%
%
\special{pn 8}%
\special{pa 4400 810}%
\special{pa 3800 1010}%
\special{fp}%
%
\special{pn 8}%
\special{pa 3800 1010}%
\special{pa 4400 1210}%
\special{fp}%
%
\special{pn 8}%
\special{pa 4400 1210}%
\special{pa 4800 1410}%
\special{fp}%
%
\special{pn 8}%
\special{pa 4800 1410}%
\special{pa 4400 1610}%
\special{fp}%
%
\special{pn 8}%
\special{pa 4400 1610}%
\special{pa 3800 1810}%
\special{fp}%
%
\special{pn 8}%
\special{pa 3800 1810}%
\special{pa 4400 2010}%
\special{fp}%
%
\special{pn 8}%
\special{pa 4400 2010}%
\special{pa 4800 1410}%
\special{fp}%
%
\special{pn 8}%
\special{sh 1.000}%
\special{ar 4400 810 40 40  0.0000000 6.2831853}%
%
\special{pn 8}%
\special{ar 4400 1010 150 340  0.0000000 0.0489796}%
\special{ar 4400 1010 150 340  0.1959184 0.2448980}%
\special{ar 4400 1010 150 340  0.3918367 0.4408163}%
\special{ar 4400 1010 150 340  0.5877551 0.6367347}%
\special{ar 4400 1010 150 340  0.7836735 0.8326531}%
\special{ar 4400 1010 150 340  0.9795918 1.0285714}%
\special{ar 4400 1010 150 340  1.1755102 1.2244898}%
\special{ar 4400 1010 150 340  1.3714286 1.4204082}%
\special{ar 4400 1010 150 340  1.5673469 1.6163265}%
\special{ar 4400 1010 150 340  1.7632653 1.8122449}%
\special{ar 4400 1010 150 340  1.9591837 2.0081633}%
\special{ar 4400 1010 150 340  2.1551020 2.2040816}%
\special{ar 4400 1010 150 340  2.3510204 2.4000000}%
\special{ar 4400 1010 150 340  2.5469388 2.5959184}%
\special{ar 4400 1010 150 340  2.7428571 2.7918367}%
\special{ar 4400 1010 150 340  2.9387755 2.9877551}%
\special{ar 4400 1010 150 340  3.1346939 3.1836735}%
\special{ar 4400 1010 150 340  3.3306122 3.3795918}%
\special{ar 4400 1010 150 340  3.5265306 3.5755102}%
\special{ar 4400 1010 150 340  3.7224490 3.7714286}%
\special{ar 4400 1010 150 340  3.9183673 3.9673469}%
\special{ar 4400 1010 150 340  4.1142857 4.1632653}%
\special{ar 4400 1010 150 340  4.3102041 4.3591837}%
\special{ar 4400 1010 150 340  4.5061224 4.5551020}%
\special{ar 4400 1010 150 340  4.7020408 4.7510204}%
\special{ar 4400 1010 150 340  4.8979592 4.9469388}%
\special{ar 4400 1010 150 340  5.0938776 5.1428571}%
\special{ar 4400 1010 150 340  5.2897959 5.3387755}%
\special{ar 4400 1010 150 340  5.4857143 5.5346939}%
\special{ar 4400 1010 150 340  5.6816327 5.7306122}%
\special{ar 4400 1010 150 340  5.8775510 5.9265306}%
\special{ar 4400 1010 150 340  6.0734694 6.1224490}%
\special{ar 4400 1010 150 340  6.2693878 6.2832853}%
%
\special{pn 8}%
\special{ar 4400 1810 150 340  0.0000000 0.0489796}%
\special{ar 4400 1810 150 340  0.1959184 0.2448980}%
\special{ar 4400 1810 150 340  0.3918367 0.4408163}%
\special{ar 4400 1810 150 340  0.5877551 0.6367347}%
\special{ar 4400 1810 150 340  0.7836735 0.8326531}%
\special{ar 4400 1810 150 340  0.9795918 1.0285714}%
\special{ar 4400 1810 150 340  1.1755102 1.2244898}%
\special{ar 4400 1810 150 340  1.3714286 1.4204082}%
\special{ar 4400 1810 150 340  1.5673469 1.6163265}%
\special{ar 4400 1810 150 340  1.7632653 1.8122449}%
\special{ar 4400 1810 150 340  1.9591837 2.0081633}%
\special{ar 4400 1810 150 340  2.1551020 2.2040816}%
\special{ar 4400 1810 150 340  2.3510204 2.4000000}%
\special{ar 4400 1810 150 340  2.5469388 2.5959184}%
\special{ar 4400 1810 150 340  2.7428571 2.7918367}%
\special{ar 4400 1810 150 340  2.9387755 2.9877551}%
\special{ar 4400 1810 150 340  3.1346939 3.1836735}%
\special{ar 4400 1810 150 340  3.3306122 3.3795918}%
\special{ar 4400 1810 150 340  3.5265306 3.5755102}%
\special{ar 4400 1810 150 340  3.7224490 3.7714286}%
\special{ar 4400 1810 150 340  3.9183673 3.9673469}%
\special{ar 4400 1810 150 340  4.1142857 4.1632653}%
\special{ar 4400 1810 150 340  4.3102041 4.3591837}%
\special{ar 4400 1810 150 340  4.5061224 4.5551020}%
\special{ar 4400 1810 150 340  4.7020408 4.7510204}%
\special{ar 4400 1810 150 340  4.8979592 4.9469388}%
\special{ar 4400 1810 150 340  5.0938776 5.1428571}%
\special{ar 4400 1810 150 340  5.2897959 5.3387755}%
\special{ar 4400 1810 150 340  5.4857143 5.5346939}%
\special{ar 4400 1810 150 340  5.6816327 5.7306122}%
\special{ar 4400 1810 150 340  5.8775510 5.9265306}%
\special{ar 4400 1810 150 340  6.0734694 6.1224490}%
\special{ar 4400 1810 150 340  6.2693878 6.2832853}%
\put(46.7000,-9.0000){\makebox(0,0){$V_1$}}%
\put(46.7000,-19.0000){\makebox(0,0){$V_3$}}%
\end{picture}%
\caption{A blow-up of $C_5 = u_0u_1u_2u_3u_4$}
\label{fig:blowup_c5}
\end{figure}

We investigate the chromatic number of 2-self-centered graphs 
(note that the chromatic number of a star is 2).
It is easy to check that 
for any integer $k \ge 2$,
there is a triangle-free $2$-self-centered graph with chromatic number exactly $k$
(see Proposition~\ref{prop:chro_exist}).
Moreover, we have the following remark by the definition of a $2$-self-centered graph.

\begin{rem}\label{rem:bip}
A graph $G$ is a bipartite $2$-self-centered graph
if and only if $G$ is isomorphic to a complete bipartite graph 
$K_{a,b}$ with $a,b \ge 2$.
\end{rem}

On the other hand, 
for every triangle-free 2-self-centered graph 
$G = GCB_X(k, \ell, \mathbb{A}_r, \mathbb{B}_s)$,
the inequalities $\chi(X) \le \chi(G) \le \chi(X)+2$ hold since $G-X$ is bipartite,
where $\chi(H)$ denotes the chromatic number of a graph $H$.
For each chromatic number within this range,
we can construct a 2-self-centered graph with that chromatic number
except the case $\chi(X)=2$, as follows.
(For the exceptional case in the following theorem, 
see Remark~\ref{rem:03} after the proof.
Note that our construction is not Mycielski's construction.)

\begin{thm}\label{prop:color1}
For any integers $p \ge 2$ and $q \in \{0,1,2\}$,
there is a triangle-free $2$-self-centered graph 
$G = GCB_X(k, \ell, \mathbb{A}_r, \mathbb{B}_s)$
such that $\chi(G) = \chi(X) + q$ with $\chi(X) = p$,
unless $p=2$ and $q=0$.
\end{thm}

Moreover, we prove a general bound for the chromatic number of 
triangle-free graphs with diameter 2
depending on its connectivity.

\begin{thm}\label{thm:chimagu}
Let $G$ be a triangle-free graph with diameter $2$ and connectivity 
$\kappa(G) = g \ge 1$.
Then $\chi(G) \le g+1$.
Furthermore, this bound is best possible.
\end{thm}

The rest of the paper is organized as follows.
In the next section, 
we describe the modification of the characterization of triangle-free 2-self-centered graphs
and give a complete proof of the correct characterization.
In Section~\ref{sec:mycielskian},
we provide a detailed explanation of Mycielskians 
to improve the self-containedness of this paper.
In Section~\ref{sec:planar},
we prove Theorem~\ref{thm:planar},
and then in Section~\ref{sec:color},
we prove Theorem~\ref{prop:color1} and Theorem~\ref{thm:chimagu}.

\section{Modification of the characterization in \cite{SMM2018}}\label{sec:characterization}

To complete the characterization,
we correctly define two notions which are introduced in~\cite{SMM2018}.
For a graph $G$, 
$V(G)$ and $E(G)$ denote the set of vertices and edges of $G$, respectively.
For a subset $S \subseteq V(G)$,
$G[S]$ denotes the subgraph of $G$ induced by $S$.
For two vertices $u,v \in V(G)$,
$d_G(u,v)$ denotes the distance between $u$ and $v$,
and for $S \subseteq V(G)$, $d_G(u,S) = \min\{d_G(u,v) : v \in S\}$.
If there is no path between $u$ and $v$,
then we define $d_G(u,v)=\infty$.
We denote by $u \sim v$ (resp., $u \not\sim v$)
if they are (resp., are not) adjacent.

\begin{defn}[Specialized bi-independent covering]\label{defn:01}
A graph $X$ is called to have a 
{\it specialized bi-independent covering via $(\mathbb{A}_r, \mathbb{B}_s)$},
an SBIC via $(\mathbb{A}_r, \mathbb{B}_s)$ in short, if the following hold:

\begin{itemize}
\item[(i)] $X$ is triangle-free.

\item[(ii)] There are two families 
$\mathbb{A}_r = \{A_1,\dots,A_r\}$
and 
$\mathbb{B}_s = \{B_1,\dots,B_s\}$
of not necessarily distinct independent subsets of $X$
such that $V(X) = \bigcup^r_{i=1} A_i = \bigcup^s_{j=1} B_j$.

\item[(iii)]
For all $u,v \in V(X)$ if $d_X(u,v) \ge 3$,
then 
there is an $1 \le i \le r$ such that $u,v \in A_i$ 
or 
there is an $1 \le j \le s$ such that $u,v \in B_j$. 

\item[(iv)]
For all $u \in V(X)$ and $i \in \{1,\dots,r\}$,
if $d_X(u,A_i) \ge 2$,
then there is $j \in \{1,\dots,s\}$ such that 
$A_i \cap B_j = \emptyset$ and $u \in B_j$.

\item[(v)]
For all $u \in V(X)$ and $j \in \{1,\dots,s\}$,
if $d_X(u,B_j) \ge 2$,
then there is $i \in \{1,\dots,r\}$ such that 
$A_i \cap B_j = \emptyset$ and $u \in A_i$.

\end{itemize} 

\end{defn}

\begin{defn}[$X$-generalized complete bipartite]\label{defn:02}
A graph $G$ is {\it $X$-generalized complete bipartite},
denoted by $GCB_X(k, \ell, \mathbb{A}_r, \mathbb{B}_s)$
with $k,\ell,r,s \ge 0$,
if $X$ has an SBIC via $(\mathbb{A}_r, \mathbb{B}_s)$ and 
$G$ is constructed in the following way:

\begin{itemize}
\item[(1)] $V(G) = K \cup L \cup Y \cup Z \cup V(X)$
where $|K|=k, |L|=\ell, Y = \{y_1,\dots,y_r\}, Z = \{z_1,\dots,z_s\}$.
Moreover, $E(X) \subseteq E(G)$.

\item[(2)] $a \sim t$ for all $a \in K$ and $t \in L \cup Y$.

\item[(3)] $b \sim t$ for all $b \in L$ and $t \in K \cup Z$.

\item[(4)] $y_i \sim t$ for all $t \in A_i$ and $1 \le i \le r$.

\item[(5)] $z_j \sim t$ for all $t \in B_i$ and $1 \le j \le s$.

\item[(6)] $y_i \sim z_j$ if and only if  $A_i \cap B_j = \emptyset$.

\end{itemize}

Some special cases must be treated separately as follows:

\begin{itemize}
\item[(7)] If $k = 0$ and $\ell \neq 0$,
then every member of $Y$ has a neighbor in $Z$
(this implies $s,r \ge 2$ by (6)).
Moreover, for all $i,j \in \{1,\dots,r\}$ with $i \ne j$,
we have $A_i \cap A_j \neq \emptyset$
or there is $p \in \{1,\dots,s\}$ such that
$A_i \cap B_p = A_j \cap B_p = \emptyset$.

\item[(8)] 
If $k \neq 0$ and $\ell = 0$, 
then every member of $Z$ has a neighbor in $Y$
(this implies $s,r \ge 2$ by (6)).
Moreover,
for all $i,j \in \{1,\dots,s\}$ with $i \ne j$,
we have $B_i \cap B_j \neq \emptyset$
or there is $q \in \{1,\dots,r\}$ such that
$B_i \cap A_q = B_j \cap A_q = \emptyset$.

\item[(9)]
If $k = \ell = 0$, then $r,s \ge 1$.
For all $i,j \in \{1,\dots,r\}$ with $i \ne j$,
we have $A_i \cap A_j \neq \emptyset$
or there is $p \in \{1,\dots,s\}$ such that
$A_i \cap B_p = A_j \cap B_p = \emptyset$.
Similarly,
for all $i,j \in \{1,\dots,s\}$ with $i \ne j$,
we have $B_i \cap B_j \neq \emptyset$
or there is $q \in \{1,\dots,r\}$ such that
$B_i \cap A_q = B_j \cap A_q = \emptyset$.

\item[(10)] $r=s=0$ if and only if 
$V(X) = \emptyset$ and $k,\ell \ge 2$.

\item[(11)] 
If $|V(X)| = 1$, then $k$ and $\ell$ are non-zero.

\end{itemize}

\end{defn}

As shown in Figure~\ref{fig:examples},
the Petersen graph can be represented as $GCB_X(1, 1, \mathbb{A}_2, \mathbb{B}_2)$,
where $X$ consists of two disjoint $K_2$s.

\begin{figure}[htb]
\centering
\input{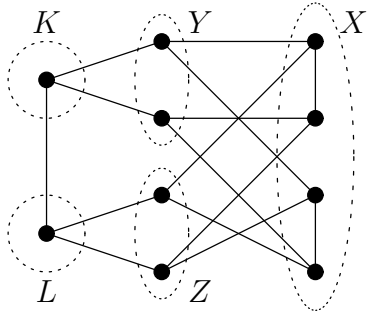}
\caption{The Petersen graph}
\label{fig:examples}
\end{figure}

\begin{rem}\label{rem:01}
Any $X$-generalized complete bipartite graph is a non-empty connected graph.
Moreover, $r=s=0$ or $r,s \ge 1$ holds,
and in particular, $G = K_{k,l}$ with $k,l \ge 2$ if $r=s=0$ by (10).
If $X \neq \emptyset$, then $r,s > 0$ by the definition, 
and hence, it is impossible for only one of them to be 0.
\end{rem}

\begin{rem}\label{rem:02}
In~\cite{SMM2018},
an $X$-generalized complete bipartite graph is not well defined.
Definition~\ref{defn:02} of our paper
includes the case of $\ell \neq 0$ and $k \neq 0$,
while Definition~11 in~\cite{SMM2018} does not.
This omission causes a problem when $k = \ell = 0$:
In this case, it is possible that no vertex in $Y$ (respectively, $Z$)
is adjacent to any vertex in $Z$ (respectively, $Y$).
(For example, $((Y \cup Z), V(X))$ may form a bipartition of a complete bipartite graph.)
Moreover, their definition also includes some unnecessary conditions.
\end{rem}

The following proposition was proved in~\cite{SMM2018};
however, 
we reprove it here since their definition has an issue, 
as noted in Remark~\ref{rem:02}.

\begin{prop}[Proposition 12 in \cite{SMM2018}]\label{prop:01}
Any $X$-generalized complete bipartite graph is 
a triangle-free $2$-self-centered graph.
\end{prop}
\begin{proof}
Let $G = GCB_X(k, \ell, \mathbb{A}_r, \mathbb{B}_s)$ and $t = |V(X)|$.
Then $n := |V(G)| = k + \ell + r + s + t$.
By Remark~\ref{rem:01},
we may assume that $r,s \ge 1$, i.e., $V(X) \neq \emptyset$.

We first show that 
any vertex $v$ in $G$ has neither degree 1 nor $n-1$.

\bigskip
\noindent
{\bf Case 1. $v \in K$ (the case $v \in L$ is similar)}

Note that $\deg(v) = \ell + r = n - k - s - t$.
Now $\deg(v) \leq n - 2$ since now $k,s \ge 1$.
On the other hand, if $\ell + r = 1$, that is, $\ell = 0$ (and $r=1$),
then now a special case (8) occurs, and hence we have $r \ge 2$, a contradiction.
Thus, $\deg(v) = \ell + r \ge 2$.

\bigskip
\noindent
{\bf Case 2. $v = y_i \in Y$ for some $i \in \{1,\dots,r\}$ (the case $v \in Z$ is similar)} 

If $\ell \neq 0$ or a vertex in $X$ is not adjacent to $v$, 
then $\deg(v) \leq n - 2$.
Thus, we assume that $\ell = 0$ and all vertices in $X$ are adjacent to $v$.
In this case,
any vertex in $Z$ is not adjacent to $v$ by (6), 
and hence $\deg(v) \leq n - 2$.
On the other hand, we assume $\deg(v) = 1$ for contradiction.
Then $v$ is adjacent to only a vertex $x \in V(X)$ and $k = 0$.
By (7) or (9), $v$ must be adjacent to a vertex in $Z$, a contradiction.

\medskip
\noindent
{\bf Case 3. $v \in V(X)$} 

If $V(X)=\{v\}$, $\deg(v) \leq n - 2$ by (11).
Suppose $|V(X)| \ge 2$.
If there is a vertex in $V(X) \setminus \{v\}$ not adjacent to $v$, then we are done.
If all vertices in $V(X) \setminus \{v\}$ are adjacent to $v$,
then since $v$ and any other vertex in $X$ cannot be in the same independent set,
there is $i \in \{1,\dots,r\}$ such that $A_i$ does not contain $v$,
that is, $y_i$ is not adjacent to $v$. Hence $\deg(v) \leq n - 2$.
On the other hand, there are $i \in \{1,\dots,r\}$ and $j \in \{1,\dots,s\}$
such that $v \in A_i$ and $v \in B_j$, and hence, $\deg(v) \ge 2$.

\bigskip
Next, we show that for any two vertices $u,v$,
they are adjacent or there is a common neighbor of them.
There are 25 different ways for choosing $u$ and $v$ from $V(G)=K \cup L \cup Y \cup Z \cup V(X)$.

If $(u,v) \in K \times L \cup K \times Y \cup L \times Z$,
then $u$ and $v$ are adjacent.

If $(u,v) \in K \times K \cup L \times L$,
then there is a path of length $2$ between $u$ and $v$ via one of the sets $L \cup Y$ or $K \cup Z$.

If $(u,v) \in Y \times Y$,
then there is a path of length $2$ between $u$ and $v$ via $K$ if $k \neq 0$.
So assume $k = 0$.
Then, by (7) or (9) together with (6),
there is a path of length $2$ between $u$ and $v$ via $X$ or $Z$.
The case $(u,v) \in Z \times Z$ is similar.

If $(u,v) \in L \times Y$,
then there is a path of length $2$ between $u$ and $v$ via $K$ if $k \neq 0$.
So assume $k = 0$.
Then, by (7),
since every vertex in $Y$ is adjacent to a vertex in $Z$,
there is a path of length $2$ between $u$ and $v$ via $Z$.
The case $(u,v) \in K \times Z$ is similar.

If $(u,v) \in Y \times Z$,
then $u = y_i$ and $v = z_j$ for some $i,j$
If $A_i \cap B_j \neq \emptyset$, 
there is a path of length $2$ between $u$ and $v$ via $X$.
Otherwise, if $A_i \cap B_j = \emptyset$, 
$u$ and $v$ are adjacent by (6).

If $(u,v) \in V(X) \times V(X)$, then we are done if $d_X(u,v) \le 2$.
If $d_X(u,v) \ge 3$, then by the definition (iii) of an SBIC via $(\mathbb{A}_r, \mathbb{B}_s)$,
there is a path of length $2$ between $u$ and $v$ via $Y$ or $Z$.

If $(u,v) \in K \times V(X) \cup L \times V(X)$,
there is a path of length $2$ between $u$ and $v$ via $Y$ or $Z$.

If $(u,v) \in Y \times V(X)$, then $u = y_i$.
If $d_X(v,A_i) \le 1$, then $d_G(u,v) \le 2$.
Otherwise, if $d_X(v,A_i) \ge 2$,
there is $j \in \{1,\dots,r\}$ such that $A_i \cap B_j = \emptyset$ and $v \in B_j$.
Thus, there is a path of length $2$ between $u$ and $v$ via $z_j \in Z$ by (6).
The case $(u,v) \in Z \times V(X)$ is similar.

\bigskip
Finally, we show that $G$ is triangle-free.
We assume to the contrary that there is a triangle $uvw$ in $G$.
Since we assume $X$ is triangle-free,
$u,v,w$ are not all together in $X$.
Moreover, at least one vertex of $u,v,w$ is contained in $X$
since $G-X$ is bipartite.
If exactly two of $u,v,w$ are in $X$, say $u,v$,
then at least one of $u,v$ is not adjacent to $w \in Y \cup Z$
since they cannot be in the same independent set.
So assume exactly one of $u,v,w$, say $w$, is in $X$.
Since both $Y$ and $Z$ are independent sets, we let $u \in Y$ and $v \in Z$.
However, this is impossible by (6). 
\end{proof}

Now we give a complete proof of Theorem~\ref{thm:ch}.

\begin{proof}[Proof of Theorem~$\ref{thm:ch}$]
The sufficiency holds by Proposition~\ref{prop:01}.
Hence we shall prove the necessity of the theorem.

Let $Y'$ be a maximal independent set of $G$,
$Z'$ a maximal independent set of $G - Y'$
and $X = G - (Y' \cup Z')$.
Let $K$ (resp., $L$) be the set of all vertices in $Z'$ (resp., $Y'$)
which are not adjacent to any member of $X$ and put $Y = Y' \setminus L$ and $Z = Z' \setminus K$.

Let $a \in K$ and $y' \in Y' = Y \cup L$.
We claim that $ay' \in E(G)$.
Suppose to the contrary that $ay' \notin E(G)$.
Since the diameter of $G$ is 2,
there is a vertex $u$ in $G$ which is adjacent to both $a$ and $y'$.
The vertex $u$ cannot be in $Y' \cup Z'$ since $Y'$ and $Z'$ are independent sets.
Hence $u \in V(X)$. However, this contradicts the definition of $K$ (property (2)).
A similar argument shows that each member of $L$ is adjacent to each member of $Z'$ (property (3)).

Let $k = |K|, \ell = |L|, r = |Y|, s = |Z|, Y = \{y_1,\dots,y_r\}$
and $Z = \{z_1,\dots,z_s\}$.
Now put $A_i = N_X(y_i)$ and $B_j = N_X(z_j)$,
where $N_X(v)$ denotes the set of neighbors of a vertex $v$ contained in $X$ (properties (4) and (5)).

Each $A_i$ and each $B_j$ are independent, since $G$ is triangle-free.
Moreover, if $y_iz_j \notin E(G)$,
then since the diameter of $G$ is 2,
there should be a vertex $x \in V(X)$ such that $y_ix, xz_j \in E(G)$.
Thus $x \in A_i \cap B_j$.
If $y_iz_j \in E(G)$,
then there must not be such a vertex $x$ since $G$ is triangle-free.
Therefore, $y_i \sim z_j$ if and only if $A_i \cap B_j = \emptyset$ (property (6)).

\medskip

We show that 
$X$ has an SBIC via $(\mathbb{A}_r, \mathbb{B}_s)$
where $\mathbb{A}_r = \{A_1,\dots,A_r\}$ and $\mathbb{B}_s = \{B_1,\dots,B_s\}$.
Let $x \in V(X)$.
Since $Y'$ and $Z'$ are maximal independent sets of $G$ and $G-Y'$, respectively,
there should be neighbors of $x$ in both $Y'$ and $Z'$.
We know that these neighbors are in $Y$ and $Z$.
By letting $y_i$ and $z_j$ be such neighbors of $x$,
we have $x \in A_i$ and $x \in B_j$.
This implies that $V(X) = \bigcup^r_{i=1} A_i = \bigcup^s_{j=1} B_j$.

It is trivial that $X$ is triangle-free, since $G$ is triangle-free.
So we check properties (iii)--(v).
If $d_X(u,v) \ge 3$ for $u,v \in V(X)$, 
then since the diameter of $G$ is 2,
there is a vertex $w \in Y \cup Z$, 
which means that $u,v \in A_i$ or $u,v \in B_j$ (property (iii)).

For $u \in V(X)$ and $i \in \{1,\dots,r\}$,
if $d_X(u,A_i) \ge 2$, i.e., $u$ is not adjacent to any vertex in $A_i$,
then there is $j \in \{1,\dots,s\}$ such that 
$u \in B_j$ and $A_i \cap B_j = \emptyset$ by $V(X) = \bigcup^s_{j=1} B_j$ and by that $G$ is 2-self-centered
(property (iv)).
We explain in more detail here:
We assume to the contrary that 
for any $j \in \{1,\dots,s\}$, 
if $u \in B_j$, then $A_i \cap B_j \neq \emptyset$.
Then $y_i$ is not adjacent to any vertex in $Z'$ by (6).
This implies that $d_G(u,y_i) \ge 3$, a contradiction.
The property (v) can be similarly proved as above.
Therefore, $X$ has an SBIC via $(\mathbb{A}_r, \mathbb{B}_s)$.

\medskip

Finally, we verify that $G$ satisfies all special cases (7)--(11).
Since (10) and (11) trivially hold by 2-self-centeredness of $G$,
we consider the cases (7), (8) and (9).
Suppose $k=0$ and $\ell \neq 0$ (i.e., the case (7)).
In this case, any vertex $y_i \in Y$ must be adjacent to a vertex in $Z$ 
since the diameter of $G$ is 2.
Moreover, the latter of (7) also holds by (4) and (6),
since $d_G(u,v) = 2$ for any two vertices $u$ and $v$ in $Y$.
The case (8) is similar.
For the case $k = \ell = 0$,
the latter of (7) and that of (8) both hold similarly to the above.
Therefore (9) is satisfied.
This completes the proof of the theorem.
\end{proof}

In the end of this section, we prove Proposition~\ref{prop:girth5}.

\begin{proof}[Proof of Proposition~$\ref{prop:girth5}$]
Let $ab$ be an edge of $G$.
We define vertex subsets of Definition~\ref{defn:02} as follows:
$K = \{a\}$, $L=\{b\}$, $Y = N_G(a) \setminus \{b\}$, 
$Z = N_G(b) \setminus \{a\}$ and $V(X) = V(G) \setminus (K \cup L \cup Y \cup Z)$
in which 
$X$ is the subgraph of $G$ induced by $V(X)$,
$Y = \{y_1,\dots,y_{r = \deg_G(a)-1}\}$,
$Z = \{z_1,\dots,z_{s = \deg_G(b)-1}\}$,
$A_i = V(X) \cap (N_G(y_i) \setminus \{a\})$
and 
$B_i = V(X) \cap (N_G(z_i) \setminus \{b\})$.
It suffices to prove that $X$ has an SBIC via $(\mathbb{A}_r, \mathbb{B}_s)$
where $\mathbb{A}_r = \{A_1,\dots,A_r\}$ and $\mathbb{B}_s = \{B_1,\dots,B_s\}$,
and all other vertex subsets satisfy conditions (1)--(6) in Definition~\ref{defn:02}.

We first check conditions (1)--(6).
It is clear that (1)--(5) hold by the above definition.
So it suffices to check that (6) holds.
If $A_i \cap B_j = \emptyset$ for some $i,j$,
then we have $y_i \sim z_j$, which contradicts the girth of $G$
since there is a 4-cycle through $y_i,z_j,L$ and $K$.
Therefore, we have 
$y_i \not\sim z_j$ and $A_i \cap B_j \neq \emptyset$ for any $i,j$,
which implies that (6) holds.

Then we show that $X$ has an SBIC via $(\mathbb{A}_r, \mathbb{B}_s)$.
By the assumption of $G$ and our definition, (i) and (ii) trivially hold.
Moreover, (iii) also holds since the diameter of $G$ is 2.
If there is a vertex $u \in V(X) \setminus A_i$ with $d_X(u,A_i) \ge 2$ for some $i$,
then $d_G(u,y_i) \ge 3$ since $y_i \not\sim z_j$ for any $i,j$, a contradiction.
Thus, there is no such vertex $u$, that is, both (iv) and (v) hold.
Therefore, $X$ has an SBIC via $(\mathbb{A}_r, \mathbb{B}_s)$,
and hence the proposition holds.
\end{proof}

\section{Mycielskian}\label{sec:mycielskian}

Let $G$ be a graph of order $n \ge 2$
and let $V(G)=\{v_1,v_2,\dots,v_n\}$.
The {\em Mycielskian} or {\em Mycielski graph} $M(G)$ of $G$
is obtained from $G$ and $n+1$ additional vertices 
$u_1, u_2, \dots, u_n, w$, as follows:
Each vertex $u_i$ is connected by an edge to $w$
and for each edge $v_iv_j \in E(G)$,
there are two edges $u_iv_j$ and $v_iu_j$
(see Figure~\ref{fig:myciel}).
Note that the resulting graph has $2n+1$ vertices and $3|E(G)|+n$ edges.
We call this construction {\em Mycielski's construction}.

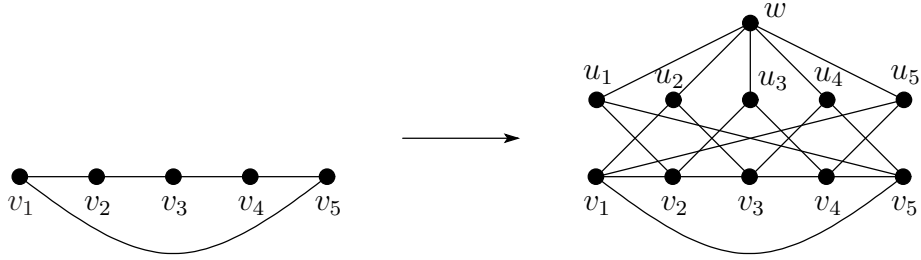
\begin{figure}[htb]
\centering
\unitlength 0.1in
\begin{picture}( 48.6000, 13.5500)( 13.9000,-36.0000)
\put(16.1500,-33.5000){\makebox(0,0){$v_1$}}%
%
\special{pn 8}%
\special{sh 1.000}%
\special{ar 1606 3200 40 40  0.0000000 6.2831853}%
%
\special{pn 8}%
\special{sh 1.000}%
\special{ar 2006 3200 40 40  0.0000000 6.2831853}%
%
\special{pn 8}%
\special{sh 1.000}%
\special{ar 2406 3200 40 40  0.0000000 6.2831853}%
%
\special{pn 8}%
\special{sh 1.000}%
\special{ar 2806 3200 40 40  0.0000000 6.2831853}%
%
\special{pn 8}%
\special{sh 1.000}%
\special{ar 3206 3200 40 40  0.0000000 6.2831853}%
%
\special{pn 8}%
\special{pa 3206 3200}%
\special{pa 1606 3200}%
\special{fp}%
\put(20.1500,-33.5000){\makebox(0,0){$v_2$}}%
\put(24.1500,-33.5000){\makebox(0,0){$v_3$}}%
\put(28.1500,-33.5000){\makebox(0,0){$v_4$}}%
\put(32.1500,-33.5000){\makebox(0,0){$v_5$}}%
%
\special{pn 8}%
\special{pa 3600 3000}%
\special{pa 4200 3000}%
\special{fp}%
\special{sh 1}%
\special{pa 4200 3000}%
\special{pa 4134 2980}%
\special{pa 4148 3000}%
\special{pa 4134 3020}%
\special{pa 4200 3000}%
\special{fp}%
%
\special{pn 8}%
\special{pa 3206 3200}%
\special{pa 3176 3222}%
\special{pa 3148 3244}%
\special{pa 3120 3264}%
\special{pa 3092 3286}%
\special{pa 3062 3306}%
\special{pa 3034 3328}%
\special{pa 3006 3348}%
\special{pa 2976 3368}%
\special{pa 2948 3388}%
\special{pa 2920 3406}%
\special{pa 2890 3424}%
\special{pa 2862 3442}%
\special{pa 2834 3460}%
\special{pa 2804 3476}%
\special{pa 2776 3492}%
\special{pa 2748 3506}%
\special{pa 2718 3520}%
\special{pa 2690 3534}%
\special{pa 2662 3546}%
\special{pa 2634 3556}%
\special{pa 2604 3566}%
\special{pa 2576 3576}%
\special{pa 2548 3582}%
\special{pa 2518 3590}%
\special{pa 2490 3594}%
\special{pa 2462 3598}%
\special{pa 2432 3600}%
\special{pa 2404 3600}%
\special{pa 2376 3600}%
\special{pa 2346 3598}%
\special{pa 2318 3594}%
\special{pa 2290 3588}%
\special{pa 2260 3582}%
\special{pa 2232 3574}%
\special{pa 2204 3566}%
\special{pa 2176 3556}%
\special{pa 2146 3544}%
\special{pa 2118 3532}%
\special{pa 2090 3520}%
\special{pa 2060 3506}%
\special{pa 2032 3490}%
\special{pa 2004 3474}%
\special{pa 1974 3458}%
\special{pa 1946 3440}%
\special{pa 1918 3422}%
\special{pa 1888 3404}%
\special{pa 1860 3386}%
\special{pa 1832 3366}%
\special{pa 1804 3346}%
\special{pa 1774 3326}%
\special{pa 1746 3304}%
\special{pa 1718 3284}%
\special{pa 1688 3262}%
\special{pa 1660 3242}%
\special{pa 1632 3220}%
\special{pa 1606 3200}%
\special{sp}%
\put(46.1500,-33.5000){\makebox(0,0){$v_1$}}%
%
\special{pn 8}%
\special{sh 1.000}%
\special{ar 4606 3200 40 40  0.0000000 6.2831853}%
%
\special{pn 8}%
\special{sh 1.000}%
\special{ar 5006 3200 40 40  0.0000000 6.2831853}%
%
\special{pn 8}%
\special{sh 1.000}%
\special{ar 5406 3200 40 40  0.0000000 6.2831853}%
%
\special{pn 8}%
\special{sh 1.000}%
\special{ar 5806 3200 40 40  0.0000000 6.2831853}%
%
\special{pn 8}%
\special{sh 1.000}%
\special{ar 6206 3200 40 40  0.0000000 6.2831853}%
%
\special{pn 8}%
\special{pa 6206 3200}%
\special{pa 4606 3200}%
\special{fp}%
\put(50.1500,-33.5000){\makebox(0,0){$v_2$}}%
\put(54.1500,-33.5000){\makebox(0,0){$v_3$}}%
\put(58.1500,-33.5000){\makebox(0,0){$v_4$}}%
\put(62.1500,-33.5000){\makebox(0,0){$v_5$}}%
%
\special{pn 8}%
\special{pa 6206 3200}%
\special{pa 6176 3222}%
\special{pa 6148 3244}%
\special{pa 6120 3264}%
\special{pa 6092 3286}%
\special{pa 6062 3306}%
\special{pa 6034 3328}%
\special{pa 6006 3348}%
\special{pa 5976 3368}%
\special{pa 5948 3388}%
\special{pa 5920 3406}%
\special{pa 5890 3424}%
\special{pa 5862 3442}%
\special{pa 5834 3460}%
\special{pa 5804 3476}%
\special{pa 5776 3492}%
\special{pa 5748 3506}%
\special{pa 5718 3520}%
\special{pa 5690 3534}%
\special{pa 5662 3546}%
\special{pa 5634 3556}%
\special{pa 5604 3566}%
\special{pa 5576 3576}%
\special{pa 5548 3582}%
\special{pa 5518 3590}%
\special{pa 5490 3594}%
\special{pa 5462 3598}%
\special{pa 5432 3600}%
\special{pa 5404 3600}%
\special{pa 5376 3600}%
\special{pa 5346 3598}%
\special{pa 5318 3594}%
\special{pa 5290 3588}%
\special{pa 5260 3582}%
\special{pa 5232 3574}%
\special{pa 5204 3566}%
\special{pa 5176 3556}%
\special{pa 5146 3544}%
\special{pa 5118 3532}%
\special{pa 5090 3520}%
\special{pa 5060 3506}%
\special{pa 5032 3490}%
\special{pa 5004 3474}%
\special{pa 4974 3458}%
\special{pa 4946 3440}%
\special{pa 4918 3422}%
\special{pa 4888 3404}%
\special{pa 4860 3386}%
\special{pa 4832 3366}%
\special{pa 4804 3346}%
\special{pa 4774 3326}%
\special{pa 4746 3304}%
\special{pa 4718 3284}%
\special{pa 4688 3262}%
\special{pa 4660 3242}%
\special{pa 4632 3220}%
\special{pa 4606 3200}%
\special{sp}%
%
\special{pn 8}%
\special{sh 1.000}%
\special{ar 6210 2800 40 40  0.0000000 6.2831853}%
%
\special{pn 8}%
\special{sh 1.000}%
\special{ar 5810 2800 40 40  0.0000000 6.2831853}%
%
\special{pn 8}%
\special{sh 1.000}%
\special{ar 5410 2800 40 40  0.0000000 6.2831853}%
%
\special{pn 8}%
\special{sh 1.000}%
\special{ar 5010 2800 40 40  0.0000000 6.2831853}%
%
\special{pn 8}%
\special{sh 1.000}%
\special{ar 4610 2800 40 40  0.0000000 6.2831853}%
%
\special{pn 8}%
\special{sh 1.000}%
\special{ar 5410 2400 40 40  0.0000000 6.2831853}%
%
\special{pn 8}%
\special{pa 5410 2400}%
\special{pa 4610 2800}%
\special{fp}%
%
\special{pn 8}%
\special{pa 5010 2800}%
\special{pa 5410 2400}%
\special{fp}%
%
\special{pn 8}%
\special{pa 5410 2400}%
\special{pa 5410 2800}%
\special{fp}%
%
\special{pn 8}%
\special{pa 5410 2400}%
\special{pa 5810 2800}%
\special{fp}%
%
\special{pn 8}%
\special{pa 6210 2800}%
\special{pa 5410 2400}%
\special{fp}%
%
\special{pn 8}%
\special{pa 5010 3200}%
\special{pa 4610 2800}%
\special{fp}%
%
\special{pn 8}%
\special{pa 5010 2800}%
\special{pa 5410 3200}%
\special{fp}%
%
\special{pn 8}%
\special{pa 5810 3200}%
\special{pa 5410 2800}%
\special{fp}%
%
\special{pn 8}%
\special{pa 5810 2800}%
\special{pa 6210 3200}%
\special{fp}%
%
\special{pn 8}%
\special{pa 6210 2800}%
\special{pa 5810 3200}%
\special{fp}%
%
\special{pn 8}%
\special{pa 5810 2800}%
\special{pa 5410 3200}%
\special{fp}%
%
\special{pn 8}%
\special{pa 5410 2800}%
\special{pa 5010 3200}%
\special{fp}%
%
\special{pn 8}%
\special{pa 5010 2800}%
\special{pa 4610 3200}%
\special{fp}%
%
\special{pn 8}%
\special{pa 4610 2800}%
\special{pa 6210 3200}%
\special{fp}%
%
\special{pn 8}%
\special{pa 6210 2800}%
\special{pa 4610 3200}%
\special{fp}%
\put(46.2000,-26.7000){\makebox(0,0){$u_1$}}%
\put(49.9000,-26.9000){\makebox(0,0){$u_2$}}%
\put(58.2000,-26.8000){\makebox(0,0){$u_4$}}%
\put(62.2000,-26.7000){\makebox(0,0){$u_5$}}%
\put(55.3000,-27.0000){\makebox(0,0){$u_3$}}%
\put(55.4000,-23.3000){\makebox(0,0){$w$}}%
\end{picture}%
\caption{Mycielski's construction applied to $C_5 = v_1v_2v_3v_4v_5$;
the right graph is the Gr\"{o}tzsch graph}
\label{fig:myciel}
\end{figure}

We denote by $M_t(G)$ 
the Mycielskian obtained from $G$ by applying 
Mycielski's construction at $t \ge 0$ times.
Observe that if $G = K_2 = M_0(G)$,
then $M_1(G) = C_5$ and $M_2(G)$ is the Gr\"{o}tzsch graph.

Mycielskians have many good properties as follows:
\begin{itemize}
\item $M(G)$ contains $G$ as an induced subgraph.
\item If the minimum degree $\delta(G)$ of $G$ is $k$, then that of $M(G)$ is $k+1$.
\item If the chromatic number of $G$ is $k$, then that of $M(G)$ is $k+1$~\cite{Mycel1955}.
\item If $G$ is a triangle-free, then so is $M(G)$~\cite{Mycel1955}.
\item If $G$ has diameter two, then so does $M(G)$~\cite{FMB1998}.
\end{itemize}

Fisher et al.~\cite{FMB1998} investigated several other propertis of Mycielskians,
such as Hamiltonicity and the domination number.
Moreover, 
there is a necessary and sufficient condition on the connectivity of Mycielskians, as follows.

\begin{thm}[Balakrishnan and Raj~\cite{BR2008}]\label{thm:myciel_conn}
Let $G$ be a connected graph.
Then $\kappa(M(G)) = \kappa(G)+1$ if and only if $\delta(G) = \kappa(G)$,
where $\kappa(G)$ denotes the vertex connectivity of $G$.
\end{thm}

\section{Proof of Theorem~\ref{thm:planar}}\label{sec:planar}


Since the sufficiency immediately holds, 
we only have to prove the necessity using Theorem~\ref{thm:ch}.
Let $G = GCB_X(k,\ell,\mathbb{A}_r, \mathbb{B}_s)$ 
be a triangle-free $2$-self-centered planar graph.
If $V(X) = \emptyset$ (i.e., $Y = Z = \emptyset$ by (10) of Definition~\ref{defn:02}),
then $G$ is a complete bipartite graph, that is, 
$G \cong K_{2,\ell}$ for $\ell \ge 2$ 
(in which we assume $|K|=2$ by symmtery).
Then we may assume that $V(X) \neq \emptyset$, that is, $|Y| \ge 1$ and $|Z| \ge 1$.

\begin{clm}\label{clm:no5cycle}
If $G$ has no $5$-cycle, then $G$ is isomorphic to a complete bipartite graph.
\end{clm}
\begin{proof}
Since the diameter of $G$ is two,
every odd cycle of length at least 7 contains a chord.
This implies that $G$ is bipartite,
since now $G$ has neither 3-cycle nor 5-cycle.
Thus, by Remark~\ref{rem:bip}, the claim holds since $G$ is planar.
\end{proof}

By Claim~\ref{clm:no5cycle}, we may assume that $G$ has a 5-cycle $C = abcde$.
Observe that 
if $K$ or $L$ is not empty,
$G$ has a 5-cycle containing at least one vertex in $K \cup L$ 
unless $G$ is a complete bipartite graph.
Thus, 
there are eleven cases depending on whether $K$ and $L$ are empty
and on which of $V(X),Y,Z,K$ and $L$ the vertices belong to:

\begin{description}
\item[The case $K \ne \emptyset$ or $L \ne \emptyset$] 
\item[(1)] $a \in K$, $b \in Y$, $c \in V(X)$, $d \in Z$, $e \in L$
\item[(2)] $a \in K$, $b,e \in Y$, $c \in V(X)$, $d \in Z$
\item[(3)] $a \in K$, $b,e \in Y$, $c,d \in V(X)$

\item[The case $K = L = \emptyset$]
\item[(i)] $a,c \in Y$, $b,d,e \in V(X)$
\item[(ii)] $a \in Y$, $c \in Z$, $b,d,e \in V(X)$
\item[(iii)] $a \in Y$, $b \in Z$, $c,d,e \in V(X)$
\item[(iv)] $a,c \in Y$, $b \in Z$, $d,e \in V(X)$
\item[(v)] $a,c \in Y$, $d \in Z$, $b,e \in V(X)$
\item[(vi)] $a \in Y$, $b,c,d,e \in V(X)$
\item[(vii)] $a,d \in Y$, $b,e \in Z$, $c \in V(X)$
\item[(viii)] $a,b,c,d,e \in V(X)$
\end{description}

\begin{figure}[thb]
\centering
\input{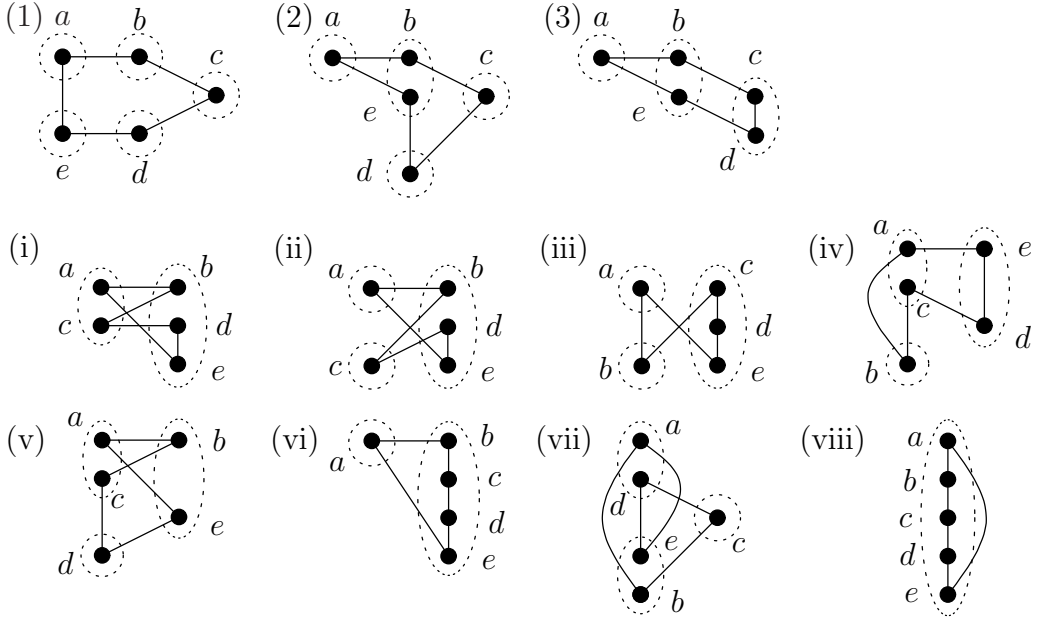}
\caption{The positions of the 5-cycle $C = abcde$}
\label{fig:allcases}
\end{figure}

We first consider (1), (2) and (3).
Hereafter, we assume that $G$ is not isomorphic to a blow-up $C_5$.

\bigskip
\noindent
{\bf Case (1).}
First we suppose by symmetry that $|K| \ge 2$.
If $|L| \ge 2$ or $|Y| \ge 2$, then $G$ clearly has a $K_{3,3}$-minor, a contradiction.
Thus, $|L|=|Y|=1$.
Since $|Y|=1$, $X$ consists only of isolated vertices.
If a vertex $z \in Z$ is not adjacent to one $x \in V(X)$,
then the distance between $z$ and $x$ is at least 3, a contradiction.
Thus, every vertex in $Z$ is adjacent to all vertices in $X$.
This implies that $|V(X)|=1$ or $|Z|=1$, that is, 
$G$ is isomorphic to a blow-up $C_5$, a contradiction.
Thus, we assume that $|K|=|L|=1$.

If $|Y|=1$ or $|Z|=1$ or $|V(X)|=1$, 
then $G$ is isomorphic to a blow-up $C_5$ similarly to the above, a contradiction.
Thus, we have $|Y| \ge 2$, $|Z| \ge 2$ and $|V(X)| \ge 2$. 
Let $y_1 \in Y \setminus \{b\}$ and $z_1 \in Z \setminus \{d\}$.
If $y_1 \sim z_1$,
then $G$ has a $K_{3,3}$-minor
since a path between $y_1$ and $d$ and one between $b$ and $z_1$ are disjoint
(see Figure~\ref{fig:find_minor}).

\begin{figure}[htb]
\centering
\unitlength 0.1in
\begin{picture}( 19.2500, 14.9500)( 13.1500,-21.1000)
%
\special{pn 8}%
\special{sh 1.000}%
\special{ar 1600 1000 40 40  0.0000000 6.2831853}%
%
\special{pn 8}%
\special{sh 1.000}%
\special{ar 1600 1800 40 40  0.0000000 6.2831853}%
\put(14.5000,-9.1000){\makebox(0,0){$a$}}%
%
\special{pn 8}%
\special{sh 1.000}%
\special{ar 2400 800 40 40  0.0000000 6.2831853}%
%
\special{pn 8}%
\special{sh 1.000}%
\special{ar 2400 1200 40 40  0.0000000 6.2831853}%
%
\special{pn 8}%
\special{sh 1.000}%
\special{ar 2400 1600 40 40  0.0000000 6.2831853}%
%
\special{pn 8}%
\special{sh 1.000}%
\special{ar 2400 2000 40 40  0.0000000 6.2831853}%
%
\special{pn 8}%
\special{sh 1.000}%
\special{ar 3200 1200 40 40  0.0000000 6.2831853}%
%
\special{pn 8}%
\special{sh 1.000}%
\special{ar 3200 1600 40 40  0.0000000 6.2831853}%
%
\special{pn 8}%
\special{pa 1600 1000}%
\special{pa 1600 1800}%
\special{fp}%
%
\special{pn 8}%
\special{pa 1600 1800}%
\special{pa 2400 2000}%
\special{fp}%
%
\special{pn 8}%
\special{pa 2400 1600}%
\special{pa 1600 1800}%
\special{fp}%
%
\special{pn 8}%
\special{pa 1600 1000}%
\special{pa 2400 800}%
\special{fp}%
%
\special{pn 8}%
\special{pa 1600 1000}%
\special{pa 2400 1200}%
\special{fp}%
%
\special{pn 8}%
\special{pa 2400 800}%
\special{pa 3200 1200}%
\special{fp}%
\special{pa 3200 1200}%
\special{pa 2400 2000}%
\special{fp}%
%
\special{pn 8}%
\special{pa 2400 1600}%
\special{pa 2400 1200}%
\special{fp}%
\put(14.7000,-19.1000){\makebox(0,0){$e$}}%
\put(22.3000,-21.1000){\makebox(0,0){$d$}}%
\put(22.3000,-7.0000){\makebox(0,0){$b$}}%
\put(32.9000,-11.1000){\makebox(0,0){$c$}}%
\put(22.4000,-12.9000){\makebox(0,0){$y_1$}}%
\put(22.7000,-17.6000){\makebox(0,0){$z_1$}}%
%
\special{pn 8}%
\special{pa 2400 1200}%
\special{pa 3200 1600}%
\special{fp}%
%
\special{pn 8}%
\special{pa 3200 1600}%
\special{pa 2400 2000}%
\special{fp}%
%
\special{pn 8}%
\special{pa 2400 1600}%
\special{pa 3200 1400}%
\special{fp}%
%
\special{pn 8}%
\special{pa 3200 1400}%
\special{pa 2400 800}%
\special{fp}%
%
\special{pn 8}%
\special{sh 1.000}%
\special{ar 3200 1400 40 40  0.0000000 6.2831853}%
\put(33.0000,-13.1000){\makebox(0,0){$w$}}%
\put(33.4000,-15.1000){\makebox(0,0){$w'$}}%
%
\special{pn 8}%
\special{ar 1600 1000 110 110  0.0000000 6.2831853}%
%
\special{pn 8}%
\special{ar 1600 1790 110 110  0.0000000 0.1090909}%
\special{ar 1600 1790 110 110  0.4363636 0.5454545}%
\special{ar 1600 1790 110 110  0.8727273 0.9818182}%
\special{ar 1600 1790 110 110  1.3090909 1.4181818}%
\special{ar 1600 1790 110 110  1.7454545 1.8545455}%
\special{ar 1600 1790 110 110  2.1818182 2.2909091}%
\special{ar 1600 1790 110 110  2.6181818 2.7272727}%
\special{ar 1600 1790 110 110  3.0545455 3.1636364}%
\special{ar 1600 1790 110 110  3.4909091 3.6000000}%
\special{ar 1600 1790 110 110  3.9272727 4.0363636}%
\special{ar 1600 1790 110 110  4.3636364 4.4727273}%
\special{ar 1600 1790 110 110  4.8000000 4.9090909}%
\special{ar 1600 1790 110 110  5.2363636 5.3454545}%
\special{ar 1600 1790 110 110  5.6727273 5.7818182}%
\special{ar 1600 1790 110 110  6.1090909 6.2181818}%
%
\special{pn 8}%
\special{ar 2400 800 110 110  0.0000000 0.1090909}%
\special{ar 2400 800 110 110  0.4363636 0.5454545}%
\special{ar 2400 800 110 110  0.8727273 0.9818182}%
\special{ar 2400 800 110 110  1.3090909 1.4181818}%
\special{ar 2400 800 110 110  1.7454545 1.8545455}%
\special{ar 2400 800 110 110  2.1818182 2.2909091}%
\special{ar 2400 800 110 110  2.6181818 2.7272727}%
\special{ar 2400 800 110 110  3.0545455 3.1636364}%
\special{ar 2400 800 110 110  3.4909091 3.6000000}%
\special{ar 2400 800 110 110  3.9272727 4.0363636}%
\special{ar 2400 800 110 110  4.3636364 4.4727273}%
\special{ar 2400 800 110 110  4.8000000 4.9090909}%
\special{ar 2400 800 110 110  5.2363636 5.3454545}%
\special{ar 2400 800 110 110  5.6727273 5.7818182}%
\special{ar 2400 800 110 110  6.1090909 6.2181818}%
%
\special{pn 8}%
\special{ar 2400 1200 110 110  0.0000000 0.1090909}%
\special{ar 2400 1200 110 110  0.4363636 0.5454545}%
\special{ar 2400 1200 110 110  0.8727273 0.9818182}%
\special{ar 2400 1200 110 110  1.3090909 1.4181818}%
\special{ar 2400 1200 110 110  1.7454545 1.8545455}%
\special{ar 2400 1200 110 110  2.1818182 2.2909091}%
\special{ar 2400 1200 110 110  2.6181818 2.7272727}%
\special{ar 2400 1200 110 110  3.0545455 3.1636364}%
\special{ar 2400 1200 110 110  3.4909091 3.6000000}%
\special{ar 2400 1200 110 110  3.9272727 4.0363636}%
\special{ar 2400 1200 110 110  4.3636364 4.4727273}%
\special{ar 2400 1200 110 110  4.8000000 4.9090909}%
\special{ar 2400 1200 110 110  5.2363636 5.3454545}%
\special{ar 2400 1200 110 110  5.6727273 5.7818182}%
\special{ar 2400 1200 110 110  6.1090909 6.2181818}%
%
\special{pn 8}%
\special{ar 2400 1600 110 110  0.0000000 6.2831853}%
%
\special{pn 8}%
\special{ar 2400 2000 110 110  0.0000000 6.2831853}%
\end{picture}%
\caption{The case when $y_1 \sim z_1$;
we have $w \ne w'$, and possibly either $w=c$ or $w'=c$.
Even if $w = c$ by symmetry,
we can find a $K_{3,3}$-minor since
$\{\{a\},\{c,d,w'\},\{z_1,w\}\}$ and $\{\{b\},\{e\},\{y_1\}\}$ form
the partite sets of the $K_{3,3}$
(three sets in each set are \emph{bags} of the $K_{3,3}$-minor).
}
\label{fig:find_minor}
\end{figure}
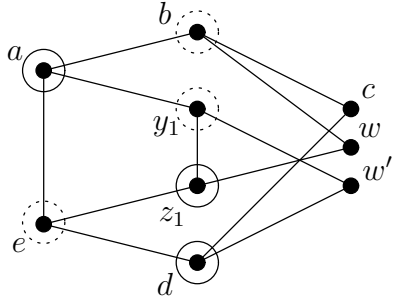

Similarly,
if a path between $y_1$ and $z_1$ and one between $b$ and $d$ are disjoint,
then it is easy to see that $G$ has a $K_{3,3}$-minor.
Therefore, 
there is a vertex $p \in V(X)$ adjacent to $y_1,b,z_1$ and $d$.

Let $q \in V(X) \setminus \{p\}$.
Assume $p \sim q$.
There are $y_2 \in Y \setminus \{b,y_1\}$ and $z_2 \in Z \setminus \{d,z_1\}$.
such that $q \sim y_2$ and $q \sim z_2$.
Then, since there is a path of length at most 2 between $y_1$ and $z_2$,
$G$ has a $K_{3,3}$-minor, a contradiction.
Thus, we assume $p \not\sim q$.
In this case,
there is a path of length at most 2 between $q$ and each of $b,y_1,d,z_1$,
where the path does not contain $p$.
Clearly, $G$ has a $K_{3,3}$-minor, a contradiction.

\bigskip
\noindent
{\bf Case (2).}
By the assumption, $|Z| \ge 2$ and $|V(X)| \ge 2$.
If $|L| \ge 1$, then this case is reduced to Case (1).
Thus, we may assume $|L|=0$. 


Since $d \sim e$ and $|Z| \ge 2$, 
there are $x \in V(X)$ and $z \in Z$  with $e \sim x$ and $z \sim x$
($x \neq c$ since $G$ is triangle-free).
Then there is a path of length at most 2 between $x$ and $b$.
First assume $x \sim b$.
There are paths $zwd$ and $zw'a$,
where $w,w' \notin \{b,d,e\}$ and possibly, $w=c$ or $w'= w \ne c$.
Thus, $G$ has a $K_{3,3}$-minor consisting of $\{a,b,c,d,e,x,z,w,w'\}$, 
a contradiction;
$\{\{a,w'\},\{c,d,w\},\{x\}\}$ and $\{\{b\},\{e\},\{z\}\}$ form
the partite sets of the $K_{3,3}$.
Hence we assume $x \not\sim b$, that is, there is a path $xwb$, possibly $w = z$.

Similarly to the above,
there are a path $zw'd$ with $w' \notin \{b,c,e,x\}$
and a path of length at most 2 between $z$ and $b$ 
(note that if $w=z$, then $z \sim b$).
Moreover, there is also a path of length at most 2 between $w'$ and $a$;
observe that the path does not contain any vertex in $\{b,c,d,e,x,z,w\}$
since $G$ is triangle-free.
Therefore, $G$ has a $K_{3,3}$-minor similarly to the above, a contradiction.

\bigskip
\noindent
{\bf Case (3).}
This case can be reduced to Case (1) or Case (2),
since there is a vertex $z \in Z$ adjacent to $d$,
and also adjacent to a vertex in $L$ or $v \in Y$ with $v \ne e$.

\bigskip

In what follows, we may assume that $K = L = \emptyset$.
Note that $|Y| \ge 2$, $|Z| \ge 2$ and $|V(X)| \ge 2$.

\begin{clm}\label{cl:reduced_cases}
Cases (i), (iv), (vi), (vii) and (viii) can be reduced to (ii), (iii) or (v).
\end{clm}
\begin{proof}
We find another 5-cycle in $G$ for each of cases (i), (iv), (vi), (vii) and (viii),
and then each case can be reduced to one of (ii), (iii) or (v),
as follows.
When we find another 5-cycle $C'$ 
and the case is reduced to a case ($*$),
we denote by $C'$ $\to$ ($*$).

\bigskip
\noindent
Case (i):
There is a vertex $z \in Z$ with $z \sim b$.
If $zd \in E(G)$, then $abzde$ $\to$ (ii).
Otherwise, 
there is a path $zwd$ where $w$ is in $Y$ or $V(X)$.
In the former case,
$cdwzb$ $\to$ (v),
and in the latter case,
$cbzwd$ $\to$ (ii) (see Figure~\ref{fig:case_i}).

\begin{figure}[htb]
\centering
\input{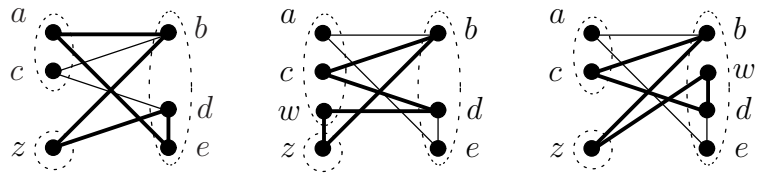}
\caption{Case (i); $zd \in E(G)$ (left), $w \in Y$ (center), $w \in V(X)$ (right)}
\label{fig:case_i}
\end{figure}

\bigskip
\noindent
Case (iv):
There is a vertex $x \in V(X)$ with $x \sim b$.
Then there is an edge $xd$ or a path $xwd$.
In the former case, $abxde \to$ (iii),
and in the latter case,
$cbxwd \to$ (iii) if $w \in V(X)$ or $bxwdc \to$ (v) if otherwise.

\bigskip
\noindent
Case (vi):
There is a vertex $z \in Z$ with $z \sim c$.
Then there is an edge $za$ or a path $zwa$ (note that $w \in V(X)$).
In the former case,
$azcde \to$ (iii),
and in the latter case,
$awzcb \to$ (ii).

\bigskip
\noindent
Case (vii):
There is $x \in V(X)$ with $x \sim a$ (since $|V(X)| \ge 2$).
If $xc \in E(G)$ or there is a path $xwc$,
then $aedcx \to$ (iv)
or $abcwx \to$ (iii) if $w \in V(X)$ 
(or $axwcb \to$ (v) if $w \in Y \cup Z$),
respectively.

\bigskip
\noindent
Case (viii): 
There is $y \in Y$ with $y \sim a$.
Then we have either $cy \in E(G)$, 
$zy,cz \in E(G)$ for some $z \in Z$
or $xy,cx \in E(G)$ for some $x \in V(X)$.
Then (first) $ycdea$ and (third) $yxcba \to$ (vii),
and (second) $yzcba \to$ (iii).

\bigskip

Therefore, each case is ultimately reduced to (ii), (iii) or (v).
\end{proof}

By Claim~\ref{cl:reduced_cases},
it suffices to consider three cases (ii), (iii) and (v).

\bigskip
\noindent
{\bf Proof of Case (ii).}
If each vertex in $Y \cup Z$ is adjacent to only $\{b,d\}$ or $\{b,e\}$,
then $G$ is isomorphic to a blow-up $C_5$, a contradiction.
Thus, since there is a cover containing $d$ in $\mathbb{A}_r$,
there is a vertex $q \in Y \setminus \{a\}$ satisfying at least one of the following by symmetry:

\begin{itemize}
\item[(C1)] $q$ is adjacent to $b,d$ and another vertex $x\in V(X)$.
\item[(C2)] $q \sim d$ but $q \not\sim b,e$
\end{itemize}

\smallskip
\noindent
{\bf (C1)} Observe that $a \sim x$ or there is a path $awx$ 
and that $c \sim x$ or there is a path $cw'x$,
where $w \notin b,d$ and $w' \notin b,d,e$ since $G$ is triangle-free
(in what follows, we do not refer to this each time).
Therefore, $G$ has a $K_{3,3}$-minor consisting of 
$\{a,b,c,d,e,x,q,w,w'\}$, a contradiction.

\bigskip
\noindent
{\bf (C2)} There are paths $awq$ and $bw'q$,
where $w \notin \{b,c,d,e\}$ and $w' \notin \{a,c,d,e,w\}$.
Moreover,
$c \sim w$ or 
there is a path $cw''w$ with $w'' \notin \{a,b,d,e,w'\}$.
Similarly to the above,
$G$ has a $K_{3,3}$-minor consisting of $\{a,b,c,d,e,q,w,w',w''\}$, a contradiction.

\bigskip
\noindent
{\bf Proof of Case (iii).}
Let $y \in Y$ be adjacent to $d$.
If there is a path $ywb$ ($w \in V(X)$),
then $ywbcd \to$ (ii).
Thus, $y \sim b$.
Similarly, 
a vertex $z \in Z$ with $z \sim d$ is adjacent to $a$.
Let $y' \in Y$ be adjacent to $c$.
Observe that 
there is a path $y'wy$ ($w \notin \{a,b,c,d,e,z\}$),
and that there is an edge $y'z$ or a path $y'w'z$ 
($w' \notin \{a,b,c,d,e,y,w\}$).
Therefore,
$G$ has a $K_{3,3}$-minor consisting of $\{a,b,c,d,e,y,z,y',w,w'\}$,
a contradiction.

\bigskip
\noindent
{\bf Proof of Case (v).}
If $a$ is not adjacent to any vertex other than $b,e$,
then $cdeab$ $\to$ (iii).
Then we may assume that 
$a$ is adjacent to $z \in Z \setminus \{d\}$ or $x \in V(X) \setminus \{b,e\}$.

First assume the former, that is, $a \sim z$ with $z \in Z \setminus \{d\}$,
and thus 
there is $p \in V(X)$ with $p \sim z$ (note that $p \ne b,e$).
Assume $c \sim p$.
Then there is an edge $pe$ or a path $pwe$ ($w \notin \{a,b,c,d,z\}$),
and there is an edge $zd$ or a path $zw'd$ ($w' \notin \{a,b,c,d,e,p,w\}$).
Therefore, 
$G$ has a $K_{3,3}$-minor consisting of $\{a,b,c,d,e,z,p,w,w'\}$,
a contradiction.
Thus $c \not\sim p$.
In this case, 
there is a path of length at most 2 between $p$ and $b$,
and also one between $p$ and $d$
(note that the intermediate vertex in those paths may be coincide).
Moreover, 
there is an edge or a path of length at most 2 between $z$ and $c$ 
that is disjoint from the above paths.
Thus, we have a similar contradiction.

Therefore, we assume the latter, that is, 
$a \sim x$ with $x \in V(X) \setminus \{b,e\}$,
and we may assume that no vertex in $Z$ is adjacent to $a$.
Assume $c \not\sim x$ and $d \not\sim x$.
In this case,
there are a path $xwc$ with $w \notin \{ a,b,e \}$
and a path $xw'd$ with $w' \notin \{ a,b,e,w \}$.
Then, 
by considering a path between $w$ and $e$ and one between $w'$ and $b$,
we have a contradiction similarly to the previous proofs.
Thus, we may assume $c \sim x$ by symmetry.
Let $z \in Z \setminus \{d\}$ with $z \sim x$.
Now there is a path $zwd$ (possibly $w = e$).
Moreover, 
there is a path $zw'b$ with $w' \notin \{ a,c,d,e,x \}$ (possibly, $w' = w \ne e$),
and hence, $G$ has a $K_{3,3}$-minor, a contradiction.

\bigskip
In each case,
$G$ is isomorphic to a complete bipartite graph or a blow-up $C_5$,
and hence the theorem holds.
\qed


\section{Chromatic number of triangle-free 2-self-centered graphs}\label{sec:color}

We first prove the following proposition.
This is already verified (by Mycielskian), 
but our proof uses another construction.

\begin{prop}\label{prop:chro_exist}
For any integer $p \ge 2$,
there is a triangle-free $2$-self-centered graph
with chromatic number $p$.
\end{prop}
\begin{proof}
If $p \in \{2,3,4\}$, 
there is a desired graph as follows:
\begin{description}
\item[$p=2$] : $K_{a,b}$ with $a,b \ge 2$ (by Remark~\ref{rem:bip})
\item[$p=3$] : $C_5$ 
\item[$p=4$] : Gr\"{o}tzsch graph
\end{description}

Thus, we assume $p \ge 5$,
and we construct a triangle-free $2$-self-centered graph 
$G = GCB_X(k, \ell, \mathbb{A}_r, \mathbb{B}_s)$ as follows.

Let $X$ be a triangle-free graph with chromatic number $p \ge 5$
(note that such a graph exists, e.g., Mycielskian).
Let $\mathcal{A} = \{A_1,A_2,\dots,A_m\}$ 
be the set of pairs of two vertices $u,v \in V(X)$ 
with ${\rm dist}(u,v) \ge 3$.
Then we define two families 
$\mathbb{A}_r = \mathcal{A} \cup \mathbb{A}'$ and $\mathbb{B}_s$ as follows:
\begin{itemize}
\item[$\mathbb{A}_r$:]
We define $\mathbb{A}' = \{A_{m+1},A_{m+2},\dots,A_{r = m+|V(X)|}\}$
such that
for each vertex $v \in V(X)$,
there is exactly one element $A_i \in \mathbb{A}'$ such that $A_i = \{v\}$.

\item[$\mathbb{B}_s$:]
We define $\mathbb{B}_s = \{B_1,B_2,\dots,B_{s= |V(X)|}\}$
such that $B_i = A_{m+i}$ for each $i = 1,\dots,|V(X)|$
(that is, $\mathbb{B}_s = \mathbb{A}'$).

\end{itemize}

Observe that $(\mathbb{A}_r, \mathbb{B}_s)$ is an SBIC of $X$:
Conditions (i) and (ii) trivially hold and (iii) holds by $\mathcal{A}$.
Since for each vertex $v$, 
there are subsets $A_i \in \mathbb{A}'$ and $B_j \in \mathbb{B}_s$
each of which contains only $v$,
(iv) and (v) also hold.

Finally, we suitably construct non-empty subsets $K,L,Y,Z$ 
to make the resultant graph $G$ be $GCB_X(k, \ell, \mathbb{A}_r, \mathbb{B}_s)$.
Therefore, in the rest of the proof,
we have only to prove that 
$G$ has a proper $p$-coloring.

Let $y_i \in Y$ (resp., $z_i \in Z$) be 
a vertex adjacent to vertices in $A_i$ (resp., $B_i$).
Let $c' : V(X) \to \{1,\dots,p\}$ be a proper $p$-coloring of $X$. 
We define a $p$-coloring $c$ of $G$ by the following rule.

\begin{enumerate}
\item $v \in V(X)$ $\Rightarrow$ $c(v) = c'(v)$.

\item $y_i \in Y$ for $i = 1,\dots,m$ 
$\Rightarrow$ $c(y_i) \in \{3,4,5\} \setminus \{c(u),c(v)\}$ where $u,v \in D_m$.  

\item $y_i \in Y$ for $i = m+1,\dots,m+|V(X)|=r$ 
$\Rightarrow$ $c(y_i) = \left\{
\begin{array}{ll}
3 &\: \mbox{if $c(u)=4$ for $u \in A_i$,} \\
4 &\: \mbox{otherwise.} 
\end{array}
\right.
$

\item $z_i \in Z$ for $i = 1,\dots,|V(X)|=s$ 
$\Rightarrow$ $c(z_i) = \left\{
\begin{array}{ll}
1 &\: \mbox{if $c(u)=2$ for $u \in B_i$,} \\
2 &\: \mbox{otherwise.} 
\end{array}
\right.
$

\item $v \in K$ $\Rightarrow$ $c(v) = 1$.
\item $v \in L$ $\Rightarrow$ $c(v) = 3$. 
\end{enumerate}

It is easy to check that the $p$-coloring $c$ is proper.
Therefore, the proposition holds.
\end{proof}

Now we prove Theorem~\ref{prop:color1}.

\begin{proof}[Proof of Theorem~$\ref{prop:color1}$]
We first assume $\chi(X) = p = 2$
Suppose $X = K_2$ where $V(X) = \{a,b\}$.
We define an SBIC via $(\mathbb{A}_2,\mathbb{B}_2)$ 
in which $\mathbb{A}_2 = \{A_1,A_2\} = \mathbb{B}_2$,
$A_1 = \{a\}$ and $A_2 = \{b\}$.
It is easy to see that 
the graph $G = GCB_X(1, 1, \mathbb{A}_2, \mathbb{B}_2)$ defined as above
has chromatic number~$\chi(X)+1 = 3$ as shown in Figure~\ref{fig:graph_ch3}.

\begin{figure}[htb]
\centering
\input{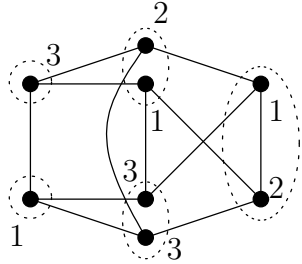}
\caption{A 3-coloring of the graph $G = GCB_X(1, 1, \mathbb{A}_2, \mathbb{B}_2)$}
\label{fig:graph_ch3}
\end{figure}

Next we assume $X = 2K_2$ where $V(X) = \{a,b,c,d\}$ and $ab, cd \in E(X)$.
We define an SBIC via $(\mathbb{A}_2,\mathbb{B}_2)$ 
in which $\mathbb{A}_4 = \{A_1,A_2,A_3,A_4\} = \mathbb{B}_4$,
$A_1 = \{a,c\}$
$A_2 = \{a,d\}$,
$A_3 = \{b,c\}$ and
$A_4 = \{b,d\}$.
We show that 
the graph $G = GCB_X(1, 1, \mathbb{A}_4, \mathbb{B}_4)$ defined as above 
has chromatic number~$\chi(X)+2 = 4$.
If $X$ is colored by only colors 1 and 2,
then there are two adjacent vertices $y \in Y$ and $z \in Z$
where both cannot be colored by color 1 and 2. 
Otherwise, without loss of generality, 
if $a,c$ are colored by color 1,
$b$ by 2 and $d$ by 3,
then $y_2,y_3,y_4 \in Y$ (corresponding to $A_2,A_3,A_4$) are colored by
color 2, 3, 1, respectively.
This means that the fourth color is required for a vertex in $K$.
Therefore, there is a triangle-free $2$-self-centered graph $G$
with $\chi(G)= p+q= 2 + q$ for each $q \in \{1,2\}$.

We now assume $\chi(X) = p \ge 3$.
By the construction in Proposition~\ref{prop:chro_exist},
we only have to construct a triangle-free $2$-self-centered graph 
for the following pairs $(p,q)$:
\begin{enumerate}
\item[(i)] $(3,0)$, $(4,0)$.
\item[(ii)] $(p,q)$ with $p \geq 3$ and $q \in \{1,2\}$.
\end{enumerate}

\bigskip
\noindent
{\bf (i)}
We first assume $p = 3$.
Let $X = C_5 = u_0u_1u_2u_3u_4$
and let $\mathbb{A}_3 = \mathbb{B}_3 = \{A_1,A_2,A_3\}$
where $A_1 = \{u_0,u_2\}$, $A_2 = \{u_1,u_3\}$ and $A_3 = \{u_2,u_4\}$.
Then we obtain $G = GCB_X(1, 1, \mathbb{A}_3, \mathbb{B}_3)$.
Let $y_i$ (resp., $z_i$) be the vertex in $Y$ (resp., $Z$) 
corresponding to $A_i$;
note that $y_1z_3, y_3z_1, y_jz_j \notin E(G)$ for each $j = 1,2,3$.
Then we have a 3-coloring $c$ of $G$ as shown in Figure~\ref{fig:case_k3}.

\begin{figure}[htb]
\centering
\input{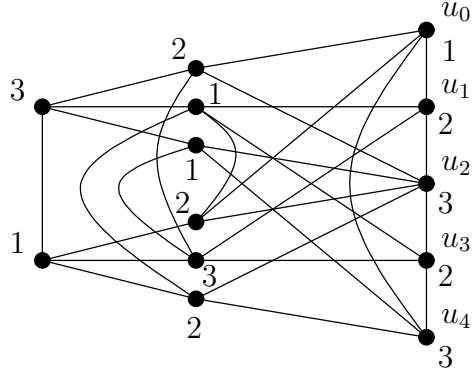}
\caption{A coloring of $G$ in the case $p=3$} 
\label{fig:case_k3}
\end{figure}


Next assume $p=4$.
Let $X$ be the Gr\"{o}tzsch graph 
(which is a triangle-free 2-self-centered graph with chromatic number 4),
and $c' : V(X) \to \{1,2,3,4\}$
be a proper 4-coloring of $X$.
Let $\mathbb{A}_4 = \mathbb{B}_4 = \{A_1,A_2,A_3,A_4\}$
where $A_i$ $(i = 1,2,3,4)$ denotes the set of vertices $v$ of $X$ with $c'(v)=i$.
Then we obtain $G = GCB_X(1, 1, \mathbb{A}_4, \mathbb{B}_4)$.
Let $y_i$ (resp., $z_i$) be the vertex in $Y$ (resp., $Z$) 
corresponding to $A_i$.
Then we have a 4-coloring $c$ of $G$ as follows.
\begin{align*}
c(v)=
\left\{
\begin{array}{ll}
c'(v) &\quad \mbox{if $v \in V(X)$} \\
1  &\quad \mbox{if $v = y_2$ or $v \in L$} \\
2  &\quad \mbox{if $v \in \{y_1,y_3,y_4\}$}\\
3  &\quad \mbox{if $v = z_4$ or $v \in K$} \\
4  &\quad \mbox{if $v \in \{z_1,z_2,z_3\}$}
\end{array}
\right.
\end{align*}

\bigskip
\noindent
{\bf (ii)}
Let $H$ be a triangle-free graph $H$ with chromatic number $p$.
We assume that 
$X$ consists of $p-2+q$ disjoint union of $H$, denoted by $H_1,\dots,H_{p-2+q}$.
We construct an SBIC via $(\mathbb{A}_r,\mathbb{B}_r)$ such that 
each element of $\mathbb{A}_r = \mathbb{B}_r$ 
contains exactly one vertex in $V(H_i)$ for each $i \in \{1,\dots,p-2+q\}$
and $|\mathbb{A}_r|=|V(H_1) \times V(H_2) \times \cdots \times V(H_{p-2+q})|$.
Since each element in $\mathbb{A}_r$ is an independent set of $X$.
we obtain $G = GCB_X(1, 1, \mathbb{A}_r, \mathbb{B}_r)$.
Let $c' : V(X) \to \{1,...,p\}$ be a proper $p$-coloring of $X$;
note that all $p$ colors appear on vertices of $H_i$ for each $i$.
Thus, we easily see that $\chi(G) \ge p + q$;
when $q=1$ (resp., $q=2$), 
if we color only $p$ (resp., $p+1$) colors of $V(G) \setminus (K \cup L)$,
all $p$ (reps., $p+1$) colors must appear in $Y$,
which means that $K$ requires one more color.

We can construct a $(p+q)$-coloring $c$ of $G$ as follows ($r \in \{1,\dots,p-q+1\}$).
\begin{align*}
c(v)=
\left\{
\begin{array}{ll}
c'(v) &\quad \mbox{if $v \in V(X)$} \\
r  &\quad \mbox{if $v \in Y$ is not adjacent to a vertex $u \in V(X)$ with $c'(u)=r$} \\
p+q  &\quad \mbox{if $v \in K \cup Z$} \\
1  &\quad \mbox{if $v \in L$} \\
\end{array}
\right.
\end{align*}
Therefore, there is a triangle-free $2$-self-centered graph $G$
with $\chi(G)= p+q$ for each pair of $p \ge 2$ and $q \in \{0,1,2\}$,
unless $p=2$ and $q=0$.
\end{proof}

\begin{rem}\label{rem:03}
If $\chi(X) = 2$, i.e., $X$ has an edge,
then there is no triangle-free $2$-self-centered graph 
$G = GCB_X(k, \ell, \mathbb{A}_r, \mathbb{B}_s)$ with $\chi(G)=2$.
(Observe that if $\chi(X) = 0$, i.e., $V(X)=\emptyset$, 
then $\chi(G)=2$ by (10) in Definition~\ref{defn:02},
and that if $\chi(X) = 1$, then $\chi(G) \in \{2,3\}$.)
Let $uv$ be an edge of $X$
and let $A_u$ and $A_v$ be independent sets in $\mathbb{A}_r$ containing $u$ and $v$,
respectively.
Then there are two vertices $y_u$ and $y_v$ in $Y$ correponding to $A_u$ and $A_v$, respectively.
Since $y_u$ and $y_v$ are not adjacent,
there is a vertex $x \neq u,v$ in $X$ adjacent to both $y_u$ and $y_v$.
Thus there is a 5-cycle $uy_uxy_vv$, and hence, $\chi(G) \ge 3$.
\end{rem}

In the end of this paper, we prove Theorem~\ref{thm:chimagu}.

\begin{proof}[Proof of Theorem~$\ref{thm:chimagu}$]
Let $S = \{u_1,u_2,\dots,u_{g}\}$ be a minimum cut set of $G$.
We color $u_i$ by color $i$ for each $i = 1,2,\dots,g$.
Let $A_i = N_G(u_i) \setminus S$.
Without loss of generality, 
$A_g$ and $A_1$ belong to distinct components in $G - S$.
Since the diameter of $G$ is exactly 2,
every vertex in $V(G) \setminus S$ belongs to some set $A_i$.
We color any vertex $v \in V(G) \setminus S$ by color $c^v_j+1$
where $c^v_j = \max \{j \in \{1,\dots,g\} : v \in A_j\}$.
This coloring is a proper coloring of $G$;
since otherwise, that is, 
if there are vertices $x,y \in V(G) \setminus S$ with the same color,
then $x,y \in A_{c^x_j}$,
that is, $x \not\sim y$ by that $G$ is triangle-free.
Furthermore, the bound is best possible;
Mycielskian obtained from $C_5$ 
by applying Mycielski's construction step $t \ge 0$ times 
is $(2+t)$-connected and its chromatic number is $3+t$ by Theorem~\ref{thm:myciel_conn}.
\end{proof}

\end{document}